\documentclass[bezier,amstex, oneside,reqno]{amsart}
\usepackage{amsmath, amssymb, amsthm, verbatim, euscript, amscd, textcomp}
\usepackage[dvips]{graphicx}
 \usepackage[all,cmtip]{xy}
\usepackage{epsfig}
\usepackage{color}
\def\ie{\emph{i.e., }}
\def\eg{\emph{e.g., }}
\def\cf{\emph{cf., }}

\def\S{\mathbb S}
\def\DD{\mathbb D}
\def\R{\mathbb R}
\def\Z{\mathbb Z}
\def\Q{\mathbb Q}

\def\D{\EuScript D}

\numberwithin{equation}{section}

\newtheorem*{problem}{Open Problem}
\newtheorem{theorem}{Theorem}[section]
\newtheorem{prop}[theorem]{Proposition}
\newtheorem{add}[theorem]{Addendum}

\newtheorem{cor}[theorem]{Corollary}
 \theoremstyle{remark}
\newtheorem{remark}[theorem]{Remark}
\newtheorem{lemma}[theorem]{Lemma}
\newtheorem*{example}{Example}
\newtheorem*{exmcg}{Example: Mapping Class Group of $\S^2\times \S^2$}

\theoremstyle{remark}

\makeatletter
\renewcommand*\env@matrix[1][*\c@MaxMatrixCols c]{%
  \hskip -\arraycolsep
  \let\@ifnextchar\new@ifnextchar
  \array{#1}}
\makeatother

\begin{document}
\author{Andrey Gogolev and Federico Rodriguez Hertz$^\ast$}
\title[Absence of Anosov diffeomorphisms]{Manifolds with higher homotopy which do not support Anosov diffeomorphisms}
\thanks{$^\ast$Both authors were partially supported by NSF grants.\\
MSC Primary 37D20; Secondary 55R10, 57R19, 14C17, 37C25}
\begin{abstract}
We show that various classes of closed manifolds with non-trivial higher homotopy groups do not support (transitive) Anosov diffeomorphisms. In particular we show that a finite product of spheres at least one of which is even-dimensional does not support transitive Anosov diffeomorphisms.
\end{abstract}
\date{}
 \maketitle

\section{Introduction}
Let $M$ be a compact smooth $n$-dimensional Riemannian manifold.
Recall that a diffeomorphism $f$ is called {\it Anosov} if there
exist constants $\mu \in (0,1)$ and $C>0$ along with a
$df$-invariant splitting $TM=E^s\oplus E^u$ of the tangent bundle of
$M$, such that for all $m \ge 0$
\begin{multline*}
\qquad\|df^mv\|\le C\lambda^m\|v\|,\;v\in E^s,\; \\
\qquad\shoveleft{\|df^{-m}v\|\le C\lambda^{m}\|v\|,\;v\in E^u.
\hfill}
\end{multline*}

The invariant distributiuons $E^s$ and $E^u$ are called the {\it stable} and {\it unstable } distributions.
If either fiber of $E^s$ or $E^u$ has dimension $k$ with $k\le\lfloor n/2\rfloor$ then $f$ is
called a {\it codimension $k$} Anosov diffeomorphism. An Anosov diffeomorphism is called {\it transitive} if there exist a point whose orbit is dense in $M$.

All currently known examples of Anosov diffeomorphisms are conjugate to algebraic automorphisms of infranilmanifolds. It is a famous open problem that dates back to Anosov and Smale to decide whether there are other examples of Anosov diffeomorphisms. In particular, Smale ~\cite{Sm} points out that it is likely that manifolds that support Anosov diffeomorphisms are covered by Euclidean spaces. Among significant partial results towards classification we mention the following two. Franks and Newhouse ~\cite{Fr, N} proved that codimension one Anosov diffeomorphisms only exist on manifolds that are homeomorphic to tori. Brin and Manning ~\cite{Br, BrM} showed that ``sufficiently pinched" Anosov diffeomorphisms only exist on infranilmanifolds.

Our purpose is to show that certain manifolds with non-trivial higher homotopy groups do not support Anosov diffeomorphisms. 
A model question was asked by \'E. Ghys in  the nineties and, surprisingly, remained unanswered: {\it does $\S^2\times\S^2$ support an Anosov diffeomorphism?} It turns out that it is easy to rule out Anosov diffeomorphisms on $\S^2\times\S^2$, as we will explain at the end of the introduction.\footnote{After finishing our paper we have received a preprint by Kleptsyn and Kudryashov~\cite{KK} that answers this particular question using similar ideas but in a more convoluted fashion.} 

Let us now present our more general results.
Many statements below have  two parts, one assuming transitivity of
the diffeomorphism and the other (a similar weaker result) not assuming transitivity.

\begin{theorem}\label{thm_product}
Let $M$ be a closed $2n$-dimensional manifold. If $E$ is the total space of a smooth sphere bundle $\S^{2n}\to E\to M$  then $E$ does not support transitive Anosov diffeomorphisms. Moreover, if all Betti numbers of $M$ are less than or equal to one then $E$ does not support Anosov diffeomorphisms. 
\end{theorem}
\begin{theorem}\label{thm_product2}
Let $M$ be a closed $n$-dimensional manifold. Assume that $m>n$. If $E$ is the total space of a smooth sphere bundle $\S^m\to E\to M$  then $E$ does not support transitive Anosov diffeomorphisms. Moreover, if $m$ is odd then $E$ does not support Anosov diffeomorphisms.
\end{theorem}
\begin{remark}
The sphere bundles we consider above are general sphere bundles with structure group $\textup{Diff}(\S^m)$, not only the ones associated to vector bundles.
\end{remark}
In particular we obtain a non-existence result for the product $M\times \S^m$.
\begin{cor}\label{cor_product}
Let $M$ be a closed $n$-dimensional manifold. If $m>n$ then the product $M\times \S^{m}$ does not support transitive Anosov diffeomorphisms. Moreover, if $m$ is odd then $M\times \S^{m}$ does not support Anosov diffeomorphisms.
\end{cor}

Under a further restriction on dimension this corollary generalizes to a different bundle setting.
\begin{theorem}\label{thm_bundle}
Let $M$ be a closed $n$-dimensional manifold. Assume that $m>n+1$. If $E$ is the total space of a smooth fiber bundle $M\to E\to \S^{m}$ then $E$ does not support transitive Anosov diffeomorphisms. Moreover, if $m$ is odd then $E$ does not support Anosov diffeomorphisms.
\end{theorem}

\begin{theorem}\label{thm_spheres}
Let $M$ be a finite product of spheres at least one of which is even-dimensional. Then $M$ does not support  transitive Anosov diffeomorphisms. 
\end{theorem}
\begin{theorem}\label{thm_spheres2}
Let $M$ be a finite product of spheres. Assume that there exists an odd dimension $k$ such that $\S^k$ appears in the product $M$ exactly once. Then $M$ does not support Anosov diffeomorphisms.
\end{theorem}
\begin{cor}
The product of an $n$-dimensional torus $\mathbb T^n$ and a $k$-dimensional sphere $\S^k$ does not support transitive Anosov diffeomorphisms if $k\ge 2$. Moreover if $k$ is odd then $\mathbb T^n\times \S^k$ does not support Anosov diffeomorphisms.
 \end{cor}
\begin{problem}
Prove that $\S^3\times\S^3$ does not support Anosov diffeomorphisms.
\end{problem}

\begin{theorem}\label{thm_euler}
Let $M$ be a compact simply connected manifold with non-zero Euler characteristic $\chi(M)$ and Betti numbers $b_i(M)\leq 2$, $i\geq 1$. Then $M$ does not support transitive Anosov diffeomorphisms. 
\end{theorem}
\begin{theorem}\label{thm_simply_connected}
Let $M$ be a closed $(2n-1)$-connected $4n$-dimensional manifold. If the $2n$-th Betti number is less than 5 then $M$ does not support Anosov diffeomorphisms.
\end{theorem}

\begin{exmcg} Consider the second cohomology group of $\S^2\times\S^2$. By the K\"unneth formula, $H^2(\S^2\times\S^2;\Z)$ is a free abelian group on two generators $x$ and $y$ that correspond to the first and the second factor. Also, by the K\"unneth formula, $x\smallsmile x=y\smallsmile y=0$ and $x\smallsmile y=\varepsilon$, where $\varepsilon$ is a generator of $H^4(\S^2\times\S^2;\Z)$. Consider a diffeomorphism $f\colon\S^2\times\S^2\to\S^2\times\S^2$ and the automorphism of $H^*(\S^2\times\S^2)$ induced by $f^2$. Clearly $(f^*)^2\varepsilon=\varepsilon$ and $(f^*)^2|_{H^2(\S^2\times\S^2;\Z)}$ is given by an integral matrix $\bigl(\begin{smallmatrix}
a & b\\
c & d
\end{smallmatrix}\bigr)$ whose determinant is one. Since $(f^*)^2$ respects the cup product we obtain the following conditions on $a, b, c$ and $d$
$$
ac=0,\; bd=0, \; ad+bc=1.
$$
It follows that $b=c=0$ and $a=d=\pm1$. Hence $f^4$ induces identity on cohomology. It follows by work of Quinn~\cite{Q} that $f^4$ is isotopic to identity. It also follows that $f$ is isotopic to a composition of $(x,y)\mapsto (y,x)$, $(x,y)\mapsto (-x,y)$ and $(x,y)\mapsto (x,-y)$.

It follows from the Lefschetz formula that an Anosov diffeomorphism cannot induce identity on cohomology. 
We will make a more extensive study of existence of Anosov diffeomorphisms on four-dimensional simply connected manifolds in a forthcoming paper.
\end{exmcg}
We remark that this argument does not work for $\S^3\times\S^3$ since the cup product anti-commutes in odd dimension. In fact, $\S^3\times\S^3$ does admit diffeomorphisms that are hyperbolic on third cohomology.

Absence of Anosov diffeomorphism on various manifolds was addressed in the literature:
\begin{enumerate}
\item 
Hirsch ~\cite{H} showed that certain manifolds with polycyclic fundamental group do not admit Anosov diffeomorphisms. In particular, he showed that mapping tori of hyperbolic toral automorphisms do not carry Anosov diffeomorphisms. 
\item
Shiraiwa ~\cite{Sh} noted that an Anosov diffeomorphism with orientable stable (or unstable) distribution cannot induce the identity map on homology in all dimensions. It follows, for example, that spheres and lens spaces do not admit Anosov diffeomorphisms.
\item 
Ruelle and Sullivan ~\cite{RS} showed that if $M$ admits a codimension $k$ transitive Anosov diffeomorphism with orientable invariant distributions then $H^k(M; \mathbb R)\neq0$. Their result is described in more detail in Section~\ref{sec_RS}.
\item
Yano ~\cite{Y} showed that negatively curved manifolds do not carry transitive\footnote{In fact, using the fact that the outer automorphism group of the fundamental group of a negatively curved manifold is finite, one can use Lefschetz formula to conclude that negatively curved manifolds do not carry Anosov diffeomorphisms} Anosov diffeomorphisms.
\end{enumerate}

{\bf Acknowledgements.} We would like to thank Tom Farrell for useful conversations. We also would like to thank the referee for careful reading and suggestions that resulted in a better exposition.

\section{Periodic points of Anosov diffeomorphisms}
Here we collect some well known facts. We refer the reader to~\cite[Chapter 6]{vick} and~\cite{Sm} for further details on this material.

Recall that if $X$ is a compact simplicial complex and  $f\colon X\to X$ is a self-map with finitely many fixed points then Lefschetz formula calculates the sum of indices of the fixed points --- the Lefschetz number --- as follows
$$
\Lambda(f)\stackrel{\mathrm{def}}{=}\sum_{p\in Fix(f)}ind_f(p)=\sum_{k\ge 0}(-1)^k Tr(f_*|_{H_k(X; \Q)}).
$$
Now assume that $X$ is a closed oriented manifold and $f$ is an Anosov diffeomorphism with oriented unstable subbundle that preserves the orientation of the unstable subbundle. Then $ind_{f^l} (x)=(-1)^{\dim E^u}$ for each $x\in Fix(f^l)$, $l\ge 1$. Hence the number of points fixed by $f^l$ can be calculated. Also, naturality of Poincar\'e duality allows us to use cohomology instead of homology.
\begin{equation}\label{lef_formula}
 |Fix(f^l)|=\left|\sum_{k\ge 0} (-1)^kTr\big((f^*)^{-l}|_{H^k(X; \Q)}\big)\right|
\end{equation}

On the other hand,  $|Fix(f^l)|$ can be calculated from the Markov coding. In particular, for a transitive Anosov diffeomorphism $f$ the following asymptotic formula holds
\begin{equation}\label{form_asymp_transitive}
|Fix(f^l)|=e^{lh_{top}(f)}+o(e^{lh_{top}(f)})
\end{equation}
where $h_{top}(f)$ is the topological entropy of $f$. 
 
 For general, not necessarily transitive Anosov diffeomorphism, $f$ formula~(\ref{form_asymp_transitive}) takes the form 
\begin{equation}\label{form_asymp_general}
|Fix(f^l)|=qe^{lh_{top}(f)}+o(e^{lh_{top}(f)})
\end{equation}
where $q$ is the number of transitive basic sets with entropy equal to $h_{top}(f)$.
Thus some of the traces in~(\ref{lef_formula}) grow exponentially fast and we obtain the following proposition.
\begin{prop}[cf. Proposition~\ref{prop1}]\label{prop0}
If manifold $M$ supports an Anosov diffeomorphism then one of the Betti numbers of $M$ is greater than one. 
\end{prop}
The following is an immediate corollary of the above proposition, Theorem~\ref{thm_product} and the Gysin exact sequence for cohomology of a sphere bundle. (For the latter see, \eg~\cite[p. 260]{Sp}.)
\begin{cor}
\label{cor1}
Let $M$ be an $\mathbb S^n$-bundle over a sphere $\mathbb S^m$. If $n\neq m$ or if $n=m$ is even then $M$ does not support Anosov diffeomorphisms.
\end{cor}

Finally, we will need the following well-known lemma which is helpful for calculating the Lefschetz number
\begin{lemma}\label{lemma_duality}
Let $M$ be a closed oriented $n$-dimensional manifold and let $f\colon M\to M$ be an orientation-preserving diffeomorphism. Choose $k\in[0,n]$ and let the induced automorphism $f_{*k}\colon H_k(M; \Q)\to H_k(M; \Q)$ be represented by a matrix $A_k$ with respect to some basis. Then the induced automorphism $f_{*(n-k)}\colon H_{n-k}(M; \Q)\to H_{n-k}(M; \Q)$ is represented by the matrix $(A_k^T)^{-1}$ with respect to the dual basis.
\end{lemma}
\begin{proof}
By the Universal Coefficients Theorem  $f^{*k}\colon H^k(M; \Q)\to H^k(M; \Q)$ is the dual of  $f_{*k}$. Hence $f^{*k}$ is represented by the transpose matrix $A_k^T$. Let $[M]$ be the fundamental class of $M$. The Poincar\'e duality isomorphism $D_k\colon H^k(M; \Q)\to H_{n-k}(M; \Q)$ given by $\varphi\mapsto[M]\smallfrown\varphi$ is natural. Therefore $f_{*(n-k)}\circ D_k\circ f^{*k}=D_k$, which implies that the matrix that represents  $f_{*(n-k)}$ is $(A_k^T)^{-1}$. 
\end{proof}

\section{Proofs of Theorems~\ref{thm_product},~\ref{thm_product2} and~\ref{thm_bundle}}

Our proofs are based on the analysis of the induced automorphism of the cohomology ring and cancellations that occur in the Lefschetz fixed point formula.

Recall that a covering $p\colon \widetilde M \rightarrow  M$ is called {\it normal} if for each $x\in M$ and each pair of lifts
$y$ and $y'$ there is a Deck transformation taking $y$ to $y'$.
The following lemma will allow us to pass to orienting covers. 

\begin{lemma} \label{lemma1}
Let $f\colon M \rightarrow M$ be a transitive Anosov diffeomorphism. Let $p\colon \widetilde M \rightarrow  M$ be a finite-sheeted normal covering. Assume that there exists a lift $\tilde f\colon\widetilde M\to\widetilde M$. Then $\tilde f$ is transitive.
\end{lemma}

\begin{proof}
Denote by $W^s$ the stable foliation of $f$. Also denote by $\mathcal F ^s$ the lift of $W^s$, {\it i.e.}, 
the stable foliation of $\tilde f$.

Recall that an Anosov diffeomorphism is transitive if and only if its stable foliation is minimal. Since $f$ is transitive, for any point $x\in M$
$$
\widetilde M = \bigcup_{y\in p^{-1}(x)} \overline{\mathcal F ^s(y)}.
$$
Since $p\colon \widetilde M \rightarrow \widetilde M$ is normal, Deck transformations act transitively on the sets $\overline {\mathcal F^s(y)}$, $y\in p^{-1}(x)$. Therefore each set $\overline {\mathcal F^s(y)}$, $y\in \widetilde M$, has a non-empty interior.

Pick a point $a\in M$. Assume that for some $b_1\in p^{-1}(a)$, $\overline {\mathcal F^s(b_1)} \neq \widetilde M$. Then, since $\widetilde M$ is connected, we can find $b_2\in p^{-1}(a)$ such that $\overline {\mathcal F^s(b_1)} \neq \overline {\mathcal F^s(b_2)}$ and $\overline {\mathcal F^s(b_1)}\cap\overline {\mathcal F^s(b_2)}\neq \varnothing$. Take a point $y\in \overline {\mathcal F^s(b_1)}\cap \overline {\mathcal F^s(b_2)}$. Then, clearly $\overline {\mathcal F^s(b_1)}\cap\overline {\mathcal F^s(b_2)}\supset \overline {\mathcal F^s(y)}$. Therefore $\overline {\mathcal F^s(b_1)}\cap \overline {\mathcal F^s(b_2)}$ has non-empty interior. Hence $\overline {\mathcal F^s(b_1)}\cap {\mathcal F^s(b_2)}\neq \varnothing$, which implies $\overline {\mathcal F^s(b_1)}\supset\overline {\mathcal F^s(b_2)}$. Similarly $\overline {\mathcal F^s(b_2)}\supset\overline {\mathcal F^s(b_1)}$. Hence  $\overline {\mathcal F^s(b_1)} = \overline {\mathcal F^s(b_2)}$ which gives a contradiction.
\end{proof}

\begin{proof}[Proof of Theorem~\ref{thm_product}] Assume that $f\colon E \to E$ is a transitive Anosov diffeomorphism. We show that without loss of generality we can assume that $E$ is fiber-oriented, oriented (hence the base $M$ is oriented) and that the unstable distribution $E^u$ is oriented.

After passing to a finite power of $f$ if necessary, pick a point $p\in E$ fixed by $f$.
Consider a finite cover $q\colon (\widetilde E, \tilde p) \rightarrow (E, p)$ that orients $E$, fiber-orients $E$ and orients the unstable distribution $E^u$. Clearly $q\colon \widetilde E \rightarrow E$ is finite-sheeted and it is easy to see that $q\colon \widetilde E \rightarrow E$ is normal. Notice that the group $q_* \pi_1 (\widetilde E,\tilde p)\subset \pi_1( E, p)$ consists of homotopy classes of loops along which $TE$, the distribution tangent to the fibers, and $E^u$ are all orientable. Therefore, it is clear that $f_*\colon \pi_1(E, p)\rightarrow \pi_1(E, p)$ preserves the conjugacy class of $q_*\pi_1(\widetilde E, \tilde p)$. This implies that $f\colon E\rightarrow E$ admits a lift $\tilde f\colon \widetilde E \rightarrow \widetilde E$ which is an Anosov diffeomorphism with orientable unstable distribution. By Lemma~\ref{lemma1}, $\tilde f$ is transitive. Thus we can assume that $E$ is oriented and fiber-oriented, and  that $E^u$ is oriented to start with.

 Also we can assume that $f$ preserves the orientation of $E^u$. Otherwise pass to $f^2$. 
 \begin{remark}
 Similar reduction will be used a few more times in the paper.
 \end{remark}
 Consider the Gysin exact sequence (see, \eg~\cite[p. 260]{Sp}) for the sphere bundle $\S^{2n}\to E\stackrel{p}{\to} M$
 $$
 \ldots \longrightarrow H^{k-(2n+1)}(M; \Z)\stackrel{0}{\longrightarrow} H^k(M; \Z) \stackrel{p^*}{\longrightarrow}H^k(E; \Z)\to H^{k-2n}(M; \Z)\longrightarrow\ldots
 $$
 The first homomorphism is zero since it is given by cupping with the Euler class of the sphere bundle $e\in H^{2n+1}(M; \Z)=0$. The second homomorphism is the pullback by the projection $p\colon E\to M$. To describe the third homomorphism we let $D$  be the mapping cylinder of $p\colon E\to M$. Then $D$ is a $(2n+1)$-dimensional disk bundle over $M$ whose boundary is $E$. The third homomorphism in the Gysin sequence is defined by the following commutative diagram
 \begin{equation}
 \label{eq_third_gysin}
 \begin{xymatrix}{
 H^k(E; \Z) \ar[d]_\Delta \ar[r] & H^{k-2n}(M; \Z) \ar[ld]^{p^*(\cdot)\smallsmile u}\\
 H^{k+1}(D,E; \Z)
 }
 \end{xymatrix}
 \end{equation}
Here $\Delta$ is the connecting homomorphism of the long exact sequence of the pair $(D, E)$, $u\in H^{2n+1}(D,E; \Z)$ is the Thom class of the oriented bundle $D\stackrel{p}{\longrightarrow} M$, and the diagonal map $z\mapsto p^*(z)\smallsmile u$ is the Thom isomorphism (see, \eg~\cite{MSt} or~\cite{Sp}).
\begin{remark}
In many sources the Gysin sequence is presented only for sphere bundles associated to vector bundles, \cf remarks in~\cite{Sp} on p. 91. However, Gysin exact sequence holds for general sphere bundles~\cite[p. 260]{Sp}. In the same way Thom's Isomorphism Theorem is usually discussed in the context of vector bundles, however a more general version for arbitrary topological disk bundles also holds true~\cite[p. 259]{Sp}.
\end{remark}

From the Gysin sequence we see that $p^*\colon H^k(M;\Z)\to H^k(E;\Z)$ is an isomorphism for $k<2n$. Also $H^k(E;\Z)\simeq H^{k-2n}(M,\Z)$ when $k>2n$. When $k=2n$ Gysin sequence yields the following short exact sequence
\begin{equation}
\label{eq_short_exact}
0\longrightarrow H^{2n}(M;\Z)\longrightarrow H^{2n}(E;\Z)\longrightarrow H^0(M;\Z)\longrightarrow 0
\end{equation}
Since $M$ is oriented $H^{2n}(M;\Z)\simeq H^0(M;\Z)\simeq \Z$ and, hence, $H^{2n}(E,\Z)\simeq\Z^2$. Denote by $\varepsilon$ the generator of $H^0(M;\Z)$ dual to the fundamental class of $M$. Consider the diagram~(\ref{eq_third_gysin}) when $k=2n$. We augment this diagram by the maps induced by the inclusion of the fiber $i\colon \S^{2n}\to E$
\begin{equation}
\label{eq_third_gysin_again}
\begin{xymatrix}{
\langle\bar y\rangle=H^{2n}(\S^{2n};\Z)\ar[d]_{\Delta'}& H^{2n}(E;\Z)\ar[l]_{i^*} \ar[r] \ar[d]_\Delta&H^{0}(M;\Z)=\langle\varepsilon\rangle \ar[ld]^{p^*(\cdot)\smallsmile u}\\
\langle i^*(u)\rangle=H^{2n+1}(\DD^{2n+1},\S^{2n};\Z) & H^{2n+1}(D, E;\Z)\ar[l]_-{i^*}
}
\end{xymatrix}
\end{equation}
It is easy to see that the connecting homomorphism $\Delta'$ is an isomorphism. The left square commutes by the naturality of long exact sequence of a pair.
By the definition of the Thom class, $i^*(u)$ is the generator of $H^{2n+1}(\DD^{2n+1},\S^{2n};\Z)$. Let $\bar y=(\Delta')^{-1}(i^*(u))$. Clearly, $p^*(\varepsilon)\smallsmile u=u$ and we see that the above diagram gives the isomorphism $H^{0}(M;\Z)\simeq H^{2n}(\S^{2n};\Z)$, $\varepsilon\mapsto \bar y$. Hence the map $H^{2n}(E;\Z)\to H^0(M;\Z)$ from the short exact sequence~(\ref{eq_short_exact}) can be viewed as $i^*\colon H^{2n}(E;\Z)\to H^{2n}(\S^{2n};\Z)$ and the short exact sequence~(\ref{eq_short_exact}) becomes
$$
0\longrightarrow H^{2n}(M;\Z)\stackrel{p^*}{\longrightarrow} H^{2n}(E;\Z)\stackrel{i^*}{\longrightarrow} H^{2n}(\S^{2n};\Z)\longrightarrow 0.
$$

We let $y=(i^*)^{-1}(\bar y)$ and then complete it to a basis $\{x,y\}$ of $H^{2n}(E;\Z)$. We identify $H^{2n}(E;\Z)$ with $\Z^2=\langle x,y\rangle$. Recall that, modulo Poincar\'e duality, cup product pairing coincides with the intersection pairing. Therefore we have $y\smallsmile y=0$ and $x\smallsmile y=\omega$, where $\omega$ is a generator of $H^{4n}(E,\Z)$. Let $x\smallsmile x=q\omega$.

Diffeomorphism $f$ induces an automorphism of $H^{2n}(E;\Z)$ given by a matrix 
$A=\bigl(\begin{smallmatrix}
  a & b\\
 c&  d                                                                                                                               
 \end{smallmatrix}
\bigr).$
After passing to $f^2$ if necessary, we can assume that $\det A=1$ and that $f^*\omega=\omega$. Since $f^*$ respects the cup product we have $(f^*y)^2=0$, $f^*x\smallsmile f^*y=\omega$ and $(f^*x)^2=q\omega$. We obtain the following equations on $a$, $b$, $c$ and $d$
$$
\begin{cases}
ad-bc=1\\
c^2q+2cd=0\\
acq+ad+bc=1\\
a^2q+2ab=q
\end{cases}
$$
It follows easily that $b=c=0$ and $a=d=\pm 1$. After passing to $f^2$ if necessary, we can assume that $a=d=1$. In particular, $f^*y=y$.
\begin{lemma}\label{lemma_cd}
Assume that $2n<k<4n$. Then the diagram
$$
\begin{xymatrix}{
H^{k-2n}(D;\Z)\ar[d]^{\smallsmile u} \ar[r]^{i^*} & H^{k-2n}(E;\Z)\ar[d]^{ \smallsmile y} \\
H^{k+1}(D,E;\Z) & H^{k}(E;\Z)\ar[l]_-\Delta
}
\end{xymatrix}
$$
commutes up to sign.
\end{lemma}
\begin{proof}
It is clear from~(\ref{eq_third_gysin_again}) that $\Delta(y)=u$. Take a cocycle $\tilde y\in C^{2n}(E;\Z)$ that represents $y$ and pick a cochain $\hat y\in C^{2n}(D,\Z)$  such that $i^\#\hat y=\tilde y$. Then $\delta\hat y$ is a cocycle in $C^{2n+1}(D, E;\Z)$ that represents $u$.

Now take any class $z\in H^{k-2n}(D;\Z)$ and a cocycle $\tilde z$ that represents $z$. Consider the short exact sequence of chain complexes
$$
\begin{xymatrix}
{
0\ar[r] & C^k(D,E;\Z)\ar[r]\ar[d]&C^k(D,\Z)\ar[r]^{i^\#}\ar[d]_\delta & C^k(E;\Z)\ar[r]\ar[d]&0\\
0\ar[r] & C^{k+1}(D,E;\Z)\ar[r]& C^{k+1}(D;\Z)\ar[r]& C^{k+1}(E,\Z)\ar[r] &0
}
\end{xymatrix}
$$
In order to calculate $\Delta(i^*(z)\smallsmile y)$ we need to go from the right-upper corner to the left-lower corner of this diagram.
Clearly $i^\#(\tilde z\smallsmile \hat y)=i^\#\tilde z\smallsmile  \tilde y$. The coboundary can be calculated as follows $\delta(\tilde z\smallsmile\hat y)=\delta\tilde z\smallsmile \hat y+(-1)^k\tilde z\smallsmile\delta\hat y=(-1)^k\tilde z\smallsmile\delta\hat y$. Since $\delta\hat y$ vanishes on chains in $E$, $\tilde z\smallsmile\delta\hat y$ also must vanish on $E$. Hence $\tilde z\smallsmile\delta\hat y$ can be viewed as a cocycle in $C^{k+1}(D,E;\Z)$. We have
$$
\Delta([i^\#\tilde z\smallsmile y])=(-1)^k[\tilde z\smallsmile\delta\hat y]=(-1)^k[\tilde z]\smallsmile[\delta\hat y]=z\smallsmile u.
$$
\end{proof}
\begin{lemma} Assume that $2n<k<4n$. Then $H^{k-2n}(E;\Z)\ni z\mapsto z\smallsmile y\in H^k(E;\Z)$ is an isomorphism.
\end{lemma}
\begin{proof}
 By combining~(\ref{eq_third_gysin}) and the diagram from Lemma~\ref{lemma_cd} we obtain the following diagram that commutes up to sign
$$
\begin{xymatrix}
{
H^k(E;\Z)\ar[r] \ar@/^2pc/[rrr]^\Delta&H^{k-2n}(M;\Z)\ar[rr]^{\text{Thom isom.}} \ar[ld]_{p^*} \ar[d]_{p^*} & &H^{k+1}(D,E;\Z)\\
H^{k-2n}(E;\Z) \ar[u]_{\smallsmile y}& H^{k-2n}(D;\Z) \ar[l]_{i^*}\ar[rru]_{\smallsmile u}
}
\end{xymatrix}
$$
Hence $z\mapsto z\smallsmile y$ is an isomorphism whose inverse is the composition (up to sign) of the isomorphisms from the Gysin sequence $H^k(E;\Z)\longrightarrow H^{k-2n}(M;\Z)\stackrel{p^*}{\longrightarrow} H^{k-2n}(E,\Z)$.
\end{proof}

Since $f^*y=y$ the following diagram commutes
$$\begin{CD}
 H^{k-2n}(E; \Z)@>f^*>>H^{k-2n}(E; \Z)\\
 @VV{\smallsmile y}V                                      @VV{\smallsmile y}V\\
 H^{k}(E; \Z)@>f^*>>H^{k}(E; \Z)
\end{CD}$$
for $k\in[2n+1,4n-1]$.
In particular, $Tr\big((f^*)^{-1}|_{H^k(E; \Q)}\big)=Tr\big((f^*)^{-1}|_{H^{k-2n}(E; \Q)}\big)$. Hence, the Lefschetz formula~(\ref{lef_formula}) takes the form
\begin{equation}\label{eq_lef_thm1}
|Fix(f^l)|=\left|2\sum_{k=2n+1}^{4n-1} (-1)^kTr\big((f^*)^{-l}|_{H^k(E; \Q)}\big)+2+Tr(A^{-l})\right|.
\end{equation}
Let
$$
\lambda\stackrel{\mathrm{def}}{=}\max_k\max\{|\mu|\,: \mu\;\;\;\mbox{is an eigenvalue of}\;\;\;(f^*)^{-1}|_{H^k(E; \Q) }\}.
$$
Because the number of periodic points grows exponentially, $\lambda$ must be greater than one. Recall that $A=Id_{\Z^2}$. Hence~(\ref{eq_lef_thm1}) implies the following asymptotic formula for the number of periodic points
\begin{equation}\label{lef_asymp}
|Fix(f^l)|=2q\lambda^l+o(\lambda^l)
\end{equation}
This contradicts to~(\ref{form_asymp_transitive}) according to which the coefficient by the leading exponential term must be one.

In the case when all Betti numbers of $M$ are less than or equal to one we have that all Betti numbers of $E$ in dimensions different from $2n$ are also less than or equal to one. Then~(\ref{eq_lef_thm1}) implies that $|Fix(f^l)|$ is uniformly bounded and, hence, diffeomorphism $f$ cannot be Anosov.
\end{proof}

Also by comparing~(\ref{form_asymp_general}) and~(\ref{lef_asymp}) we get the following statement.
\begin{add}[to Theorem~~\ref{thm_product}]
Let $M$ be an orientable closed $2n$-dimensional manifold. Assume that $\S^{2n}\to E\to M$ is an oriented smooth sphere bundle over $M$. Assume that  $f\colon E \to E$ is an Anosov diffeomorphism whose unstable distribution is orientable. Then the number of basic sets that carry maximal topological entropy $h_{top}(f)$ is even.
\end{add}
Analogous addenda to Theorems \ref{thm_product2},~\ref{thm_bundle} and~\ref{thm_spheres} also hold.
\begin{proof}[Proof of Theorem~\ref{thm_bundle}]
 Assume that there exists an Anosov diffeomorphism $f\colon E\to E$. As earlier, we can assume that $E$ is oriented and that $f$ is orientation preserving. It is easy to calculate cohomology of $E$, \eg from the Wang exact sequence (see \eg~\cite[p. 456]{Sp}) we obtain
 $$
 H^k(E; \Q)=
 \begin{cases}
  H^k(M; \Q) & k=0,\ldots n \\
  0 & k=n+1,\ldots m-1 \\
  H^{k-m}(M; \Q) & k=m,\ldots m+n
 \end{cases}
 $$
The isomorphism $i^*\colon H^k(E; \Q)\to H^k(M; \Q)$, $k\le n$, is induced by the inclusion of the fiber $i\colon M\to E$. Also $i$ induces an isomorphism on homology $i_*\colon H_k(M; \Q)\to H_k(E; \Q)$, $k\le n$.

Fix an integer $k\in[0,n]$. Denote by $f^{*k}$ and $f_{*k}$ the automorphisms induced by $f$ on $H^k(E; \Q)$ and $H_k(E; \Q)$ respectively. By Lemma~\ref{lemma_duality}, if  $f_{*k}$ is represented by a matrix $A_k$ then   the matrix that represents  $f_{*(m+n-k)}$ is $(A_k^T)^{-1}$. 

Let $D_k^M\colon H^k(M; \Q)\to H_{n-k}(M; \Q)$ be the Poincar\'e duality isomorphism for $M$. Consider the isomorphism $D_k'\colon H^k(E; \Q)\to H_{n-k}(E; \Q)$ given by the composition
$$
\begin{CD}H^k(E; \Q)@>i^*>>H^k(M; \Q)@>D_k^M>>H_{n-k}(M; \Q)@>i_*>> H_{n-k}(E; \Q).
\end{CD}
$$
It follows from naturality of the cap product that $D_k'$ is given by $\varphi\mapsto i_*[M]\smallfrown\varphi$, where $[M]$ is the fundamental class of $M$. 

Next we check that $D_k'$ is also natural. Let $\varphi\in H^k(E; \Q)$. Then
$$
f_{*(n-k)}\big(D_k'(f^{*k}\varphi)\big)
=f_{*(n-k)}(i_*[M]\smallfrown f^{*k}\varphi)\stackrel{(*)}{=}f_{*n}(i_*[M])\smallfrown\varphi\stackrel{(**)}{=}i_*[M]\smallfrown\varphi.
$$
Here $(*)$ is due to naturality of cap product and $(**)$ is because $H_n(E;\Q)$ is isomorphic to $\Q$ and hence we can assume that $f_{*n}$ is identity.

The above calculation shows that $D_k'=f_{*(n-k)}D_k'f^{*k}$. Therefore $f_{*(n-k)}$ is represented by the matrix $(A_k^T)^{-1}$. Now we let $l\ge 1$ and apply the Lefschetz formula to $f^l$. We perform certain cancellations using the observations that we have made.
\begin{multline*}
|Fix(f^l)|=\Bigl|\sum_{k=0}^{m+n}(-1)^kTr(f_{*k}^l)\Bigr|\\
=\Bigl|\sum_{k=0}^{n}(-1)^kTr(f_{*k}^l)+\sum_{k=0}^{n}(-1)^{m+n-k}Tr(f_{*(m+n-k)}^l)\Bigr|\\
=\Bigl|\sum_{k=0}^{n}(-1)^kTr(A_{k}^l)+\sum_{k=0}^{n}(-1)^{m+n-k}Tr((A_k^T)^{-l})\Bigr|\\
=\Bigl|\sum_{k=0}^{n}(-1)^kTr(A_{k}^l)+\sum_{k=0}^{n}(-1)^{m+n-k}Tr(A_{n-k}^{l})\Bigr|\\
=\Bigl|\sum_{k=0}^{n}(-1)^kTr(A_{k}^l)+\sum_{k=0}^{n}(-1)^{m+k}Tr(A_{k}^{l})\Bigr|\\
=\Bigl|\sum_{k=0}^{n}\big((-1)^k+(-1)^{m+k}\big)Tr(A_{k}^l)\Bigr|.
\end{multline*}
We see that if $m$ is odd then $Fix(f^l)=\varnothing$, which gives a contradiction. If $m$ is even then the above calculation gives the asymptotic formula
$$
|Fix(f^l)|=2q\lambda^l+o(\lambda^l),
$$
for some $q\in\Z$ and $\lambda>1$. Together with~(\ref{form_asymp_transitive}) this implies that $f$ is not a transitive Anosov diffeomorphism.
\end{proof}
\begin{proof}[Proof of Theorem~\ref{thm_product2}]
Assume that there exists an Anosov diffeomorphism $f\colon E\to E$. Again, by passing to a finite cover if necessary, we can assume that $M$ is orientable and $E$ is orientable. And by passing to $f^2$ if necessary, we can assume that $f$ preserves orientation. Cohomology of $E$ can be easily computed (\eg using the Gysin exact sequence):
$$
 H^k(E; \Q)=
 \begin{cases}
  H^k(M; \Q) & k=0,\ldots n \\
  0 & k=n+1,\ldots m-1 \\
  H^{k-m}(M; \Q) & k=m,\ldots m+n
 \end{cases}
$$
Since the dimension of the fiber is greater than the dimension of the base, the bundle $p\colon E\to M$ admits a section $s\colon M\to E$, $p\circ s=id_M$. Again, from the Gysin sequence, $p^*\colon H^k(E; \Q)\to H^k(M; \Q)$ is an isomorphism for $k=0,\ldots n$. It follows that $s^*\colon H^k(E; \Q)\to H^k(M; \Q)$ is an isomorphism for $k=0,\ldots n$. 

Section $s$ plays the role of the inclusion $i$ from the proof of Theorem~\ref{thm_bundle} and the rest of the proof proceeds in exactly the same way as the proof of Theorem~\ref{thm_bundle}. 
\end{proof}

\section{Proofs of Theorems~\ref{thm_spheres} and~\ref{thm_spheres2}}
Let $M=(\S^{d_1})^{n_1}\times(\S^{d_2})^{n_2}\times\ldots \times(\S^{d_m})^{n_m}$, where $d_1<d_2<\ldots <d_m$ are the dimensions of the spheres. If $a$ is the generator of $H^{d_p}(\S^{d_p}; \Z)$ and $\pi\colon M\to\S^{d_p}$ is the projection to one of the factors then $x=\pi^*(a)$ is a cohomology class in $H^{d_p}(M; \Z)$. By the K\"{u}nneth formula (see, \eg~\cite[p. 247]{Sp}) such classes generate $H^*(M; \Z)$. Introduce the following notation for these generators
$$
x_1^1, x_2^1, \ldots\quad x_{n_1}^1; x_1^2, x_2^2, \ldots\quad x_{n_2}^2; \ldots\quad  x_1^m, x_2^m, \ldots\quad x_{n_m}^m.
$$
Then the degree of $x_q^p$ is $d_p$. These classes are subject to relations
$$
x_q^p\smallsmile x_q^p=0, \quad\quad x_q^p\smallsmile x_l^k=(-1)^{d_p d_k}x_l^k\smallsmile x_q^p.
$$
Also note that we have chosen an increasing order on the generators. 

Fix $d\geq 1$ and consider $H^d(M; \Z)$ as a free abelian group. We will choose an ordered basis for $H^d(M; \Z)$ which we will then use throughout the proofs. Let $\vec\alpha=(\alpha_1, \alpha_2, \ldots \alpha_m)$ be an $m$-tuple of integers such that $0\leq \alpha_p\leq n_p$, $p=1, \ldots, m$, and $d=\alpha_1d_1 + \alpha_2d_2+\ldots +\alpha_md_m$. Call such an $m$-tuple a {\it splitting of} $d$. For each $p$ pick an $\alpha_p$-tuple of generators of degree $d_p$
$$
x^p_{i(1)}, x^p_{i(2)}, \ldots \quad x^p_{i(\alpha_p)}, \quad i(1)<i(2)<\ldots<i(\alpha_p).
$$

If $\alpha_p=0$ then let $y(\alpha_p)=1\in H^0(M; \Z)$. Otherwise let $y(\alpha_p)=x^p_{i(1)}\smallsmile x^p_{i(2)}\smallsmile \ldots \quad \smallsmile x^p_{i(\alpha_p)}$. Finally let 
$$
z=y(\alpha_1)\smallsmile y(\alpha_2)\smallsmile \ldots\quad \smallsmile y(\alpha_m).
$$
Clearly $z$ is a product of generators that has degree $d$. By the K\"{u}nneth formula the collection of all classes of this form is a basis of $H^d(M; \Z)$. Given a basis element $z \in H^d(M; \Z)$ we will write $\vec\theta(z)$ for the splitting from which $z$ was obtained.

Given two splittings of $d$ --- $\vec \alpha= (\alpha_1, \alpha_2, \ldots\quad \alpha_m)$ and $\vec\beta = (\beta_1, \beta_2,\ldots\quad \beta_m)$ --- we declare that $\vec\alpha <\vec\beta$ if there exists $p\geq  1$ such that $\alpha_i =\beta_i$ for all $i<p$ and $\alpha_p <\beta_p$. This order on splittings together with the lexicographic order on the set of $\alpha_p$-tuples induces an order on the chosen basis of $H^d(M; \Z)$.

\begin{example}
Let $M=(\S^1)^4\times(\S^2)^2\times(\S^3)^2$. Then the basis of $H^3 (M; \Z)$ is $x_1^3$,  $x_2^3$, $x_1^1\smallsmile x_1^2$,  $x_1^1\smallsmile x_2^2$,  $x_2^1\smallsmile x_1^2$, $x_2^1\smallsmile x_2^2$, $x_3^1\smallsmile x_1^2$, $x_3^1\smallsmile x_2^2$, $x_4^1\smallsmile x_1^2$, $x_4^1\smallsmile x_2^2$, $x_1^1\smallsmile x_2^1\smallsmile x_3^1$, $x_1^1\smallsmile x_2^1\smallsmile x_4^1$, $x_1^1\smallsmile x_3^1 \smallsmile x_4^1$, $x_2^1\smallsmile x_3^1\smallsmile x_4^1$ .
\end{example}

Let $f^*$ be an automorphism of the cohomology ring $H^*(M;\Z)$. Let $\EuScript B^d$ be the ordered basis of $H^d (M; \Z)$ as described above. Given a splitting $\vec\alpha$ let 
$$
\EuScript C^d (\vec\alpha) = \{b\in \EuScript B^d \colon \vec\theta(b)=\vec\alpha\} 
$$
and let
$$
\EuScript B^d (\vec\alpha) = \{b\in \EuScript B^d \colon \vec\theta(b)\ge\vec\alpha\} .
$$
Clearly $\EuScript B^d(\mbox{-})=\{\EuScript B^d(\vec\alpha)\colon \vec\alpha$ is a splitting of $d \}$ is a filtration of $\EuScript B^d$.

\begin{lemma}\label{lemma_upper_triangular}
The filtration $\EuScript B^d(\mbox{-})$ spans an $f^*$-invariant filtration of $H^d (M; \Z)$. That is, for every $d$ and for every splitting $\vec\alpha$
$$
f^*span_\Z \: \EuScript B^d(\vec\alpha) = span_\Z \: \EuScript B^d(\vec\alpha).
$$
\end{lemma}
\begin{proof}
Note that the lemma becomes obvious if $d=d_p$ and $\vec\alpha$ is the splitting given by $d_p=1 \cdot d_p$. Indeed, since $\vec\alpha$
is the smallest splitting, $\EuScript B^d(\vec\alpha)=\EuScript B^d$.

Now take arbitrary $d\geq 1$ and let $\vec\alpha$ be a splitting of $d$. Take any $z=x_{i_1}^{p_1}\smallsmile x_{i_2}^{p_2}\smallsmile \ldots \smallsmile x_{i_k}^{p_k}\in \EuScript C^d(\vec\alpha)$. Then $ f^*z=f^*x_{i_1}^{p_1}\smallsmile f^*x_{i_2}^{p_2}\smallsmile \ldots \smallsmile f^*x_{i_k}^{p_k}$, because $f^*$ respects the cup product. Each factor $f^*x_{i_j}^{p_j}$ is a linear combination of basis elements from $\EuScript B^{d_{p_j}}$. It follows from the observation made at the beginning of the proof of the lemma that after distributing $f^*z$ becomes a linear combination of basis elements whose splittings are greater than or equal to $\vec\alpha$. Hence $f^*z\in span_\Z \: \EuScript B^d(\vec\alpha)$ and therefore $f^* span_\Z \:\EuScript C^d(\vec\alpha)\subset span_\Z\: \EuScript B^d(\vec\alpha)$ for every splitting $\vec\alpha$. It follows that $f^* span_\Z\: \EuScript B^d(\vec\alpha)\subset span_\Z \:\EuScript B^d(\vec\alpha)$. Applying the same reasoning for $(f^*)^{-1}$ gives the opposite inclusion 
and we obtain 
$$
f^* span_\Z\: \EuScript B^d(\vec\alpha)= span_\Z\: \EuScript B^d(\vec\alpha)
$$
for every splitting $\vec\alpha$.
\end{proof}

\begin{lemma}\label{lemma_even_id}
If $d_p$ is an even degree then there exists $l\geq 1$ such that $(f^*)^l x_q^p = x_q^p$, $q=1, \ldots n_p$.
\end{lemma}
\begin{proof}
Let $\vec\alpha$ be the splitting given by $d_p=1\cdot d_p$ and let $\vec\beta$ be the next splitting, \ie the smallest splitting of $d_p$ which is greater than $\vec\alpha$.

Fix $q\in [1, n_p]$. Write $f^*x_q^p$ in the basis of $H^d (M; \Z)$
$$
f^*x_q^p=\sum_{i=1}^{n_p} a_ix_i^p + \sum_{k}b_ky_i^p,
$$
here $y_k^p\in \EuScript B^{d_p}(\vec\beta)$, \ie $y_k^p$-s are non-trivial cup products of the generators. By Lemma~\ref{lemma_upper_triangular}
$$
(f^*)^{-1}(\sum_kb_ky_k^p)\in span_\Z \:\EuScript B^{d_p}(\vec\beta).
$$
Hence  $a_i\neq 0$ for an least one $i$.

The cup product 
$$
f^*x_q^p\smallsmile f^*x_q^p = \sum_{i<j} 2a_ia_jx_i^p\smallsmile x_j^p +\sum_{i, k} 2a_ib_kx_i^p\smallsmile y_k^p + \sum_{k<r} 2b_kb_ry_k^p\smallsmile y_r^p
$$
must vanish. Note that $x_i^p\smallsmile x_j^p \neq 0$ if $i<j$ and that $x_i^p\smallsmile y_k^p\neq 0$. It follows that $b_k=0$ for all $k$ and that there exists exactly one $i=i(q)$ such that $a_{i(q)}\neq 0$.

We conclude that $f^*(span_\Z\:\EuScript C^{d_p}(\vec\alpha))\subset span_\Z\:\EuScript C^{d_p}(\vec\alpha)$. Also recall that by Lemma~\ref{lemma_upper_triangular} $f^*(span_\Z\:\EuScript B^{d_p}(\vec\beta))\subset span_\Z\:\EuScript B^{d_p}(\vec\beta)$.  It follows that $f^*$ has block-diagonal form with one block corresponding to $\EuScript C^{d_p}(\vec\alpha)$ and the other one to $\EuScript B^{d_p}(\vec\beta)$. Hence $f^*\vert_{span_\Z\:\EuScript C^{d_p}(\vec\alpha)}$ is an automorphism of $span_\Z\:\EuScript C^{d_p}(\vec\alpha)\simeq \Z^{d_p}$. It follows that $q\mapsto i(q)$ is a permutation and $a_{i(q)}=\pm 1$. Hence the lemma holds with $l=2d_p!$.
\end{proof}

Now assume that $f\colon M\to M$ is an Anosov diffeomorphism. Let $f^*$ be the induced automorphism on $H^*(M;\Z)$. We denote by $f^{*d}$ the restriction of $f^*$ to $H^d(M; \Z)$. We identify the automorphism $f^{*d}$ with the matrix that represents it in the ordered basis of $H^d(M; \Z)$ that we have chosen earlier. 

First recall that for each $p=1, \ldots m$ the automorphism $f^{*d_p}$ has block-upper-triangular form 
$$
f^{*d_p} = \begin{pmatrix}[c|c] A_p & *\\  \hline  0 & *\end{pmatrix},
$$
where $A_p$ corresponds to the first $n_p$ basis elements $x_1^p, x_2^p,\ldots\quad x_{n_p}^p$. Clearly $\det A_p=\pm 1$. By passing to a finite power of $f$ we can assume that $\det A_p=1, p=1, \ldots m$. Also, by passing to a further finite power, we can assume, by Lemma~\ref{lemma_even_id}, that $A_p = Id$ whenever $d_p$ is even.

Now consider $f^{*d}$ for arbitrary $d$. By Lemma~\ref{lemma_upper_triangular}, $f^{*d}$ has block-upper-triangular form. Each diagonal block corresponds to $\EuScript C^d(\vec\alpha)\subset\EuScript B^d$, where $\vec\alpha$ is a splitting of $d$, $d=\sum_{i=1}^m \alpha_id_i$. We denote this block by $A(\vec\alpha)$.

We say that a splitting $\vec\alpha$ is {\it odd} and we write $\vec\alpha \in \EuScript O$ if $\alpha_p=0$ for all even $d_p$, $p=1, \ldots\quad m$. If $\alpha_p=0$ whenever $d_p$ is odd then we say that splitting $\vec\alpha$ is {\it even}. Recall that $f^*$ preserves the cup product and that the cup product anti-commutes in odd dimension. Hence, for odd $\vec\alpha$
$$
A(\vec\alpha) = \bigotimes_{1\leq p\leq m} A_p^{\wedge \alpha_m}.
$$

Here $A_p^{\wedge \alpha_m}$ is the exterior power of $A_p$ (we set $A_p^{\wedge 0}=Id_\Z$). For a general splitting $\vec\alpha$ the block $A(\vec\alpha)$ is a block-diagonal matrix itself with $\prod_{\substack{1\leq p\leq m; \\ \text{$d_p$ is even}}}\begin{pmatrix} n_p \\ \alpha_p \end{pmatrix}$ identical blocks given by
\begin{equation}
 \label{block}
B(\vec\alpha) = \bigotimes_{\substack{1\leq p\leq m; \\ \text{$d_p$ is odd}}} A_p^{\wedge \alpha_p}
\end{equation}
Notice that if $\vec\alpha$ is even then $B(\vec\alpha)$ is simply a one-by-one identity matrix.

Denote by $e$ the total number of generators of even degree, \ie
$$
e=\sum_{\substack{1\leq p\leq m; \\ \text{$d_p$ is even}}} n_p.
$$

The above observations imply the following statement.
\begin{lemma}\label{lemma_blocks}
Matrices $f^{*d}$, $d=0, 1, \ldots \quad \dim M$, have block-upper-triangular form. Every diagonal block in $f^{*d}$ is either one-by-one identity matrix or $A(\vec\alpha)$, where $\vec\alpha$ is an odd splitting. For each odd splitting $\vec\alpha$ the block $A(\vec\alpha)$ appears $2^e$ times as a diagonal block in $f^{*d}$, $d=0, 1, \ldots\quad \dim M$. Moreover, if a block $A(\vec\alpha)$ appears in $f^{*d'}$ and $f^{*d''}$ then $d'$ and $d''$ have the same parity.
\end{lemma}
If $\vec\alpha$ is an odd splitting and block $A(\vec\alpha)$ appears in $f^{*d}$ then we denote the parity of $d$ by $\varepsilon (\vec\alpha)$.
\begin{example}
Let $M=(\S^1)^2\times(\S^2)^2\times(\S^3)^2$. Recall that $A_1\colon \Z^2\to\Z^2$ is the automorphism of $H^1(M; \Z)$ and $A_3\colon \Z^2\to\Z^2$
is the first block of the automorphism $f^{*3}$ coming from $(\S^3)^2$. Then we have the following formulae for $f^{*d}$, $d=0, \ldots\quad 12$,
\begin{align*}
&f^{*0}=Id_\Z, \\
&f^{*1}=A_1,\\
&f^{*2}= \text{upp.tr}(A_1^{\wedge2}, Id_{\Z^2})\stackrel{\mathrm{def}}{=}\begin{pmatrix}[c|c] A_1^{\wedge 2} & *\\ \hline 0 & Id_{\Z^2}\end{pmatrix},\\
&f^{*3}=\text{upp.tr}(A_3, A_1, A_1), \\
&f^{*4}=\text{upp.tr}(A_1\otimes A_3, Id_\Z, A_1^{\wedge 2}, A_1^{\wedge 2}),\\
&f^{*5}=\text{upp.tr}(A_3, A_3, A_1^{\wedge 2}\otimes A_3, A_1),\\
&f^{*6}=\text{upp.tr}(A_3^{\wedge 2}, A_1\otimes A_3, A_1\otimes A_3, A_1^{\wedge 2}),\\
&f^{*7}=\text{upp.tr}(A_1\otimes A_3^{\wedge 2}, A_3, A_1^{\wedge 2}\otimes A_3, A_1^{\wedge 2}\otimes A_3),\\
&f^{*8}=\text{upp.tr}(A_3^{\wedge 2}, A_3^{\wedge 2}, A_1^{\wedge 2}\otimes A_3^{\wedge 2}, A_1\otimes A_3),\\
&f^{*9}=\text{upp.tr}(A_1\otimes A_3^{\wedge 2}, A_1\otimes A_3^{\wedge 2}, A_1^{\wedge 2}\otimes A_3),\\
&f^{*10}=\text{upp.tr}(A_3^{\wedge 2}, A_1^{\wedge 2}\otimes A_3^{\wedge 2}, A_1^{\wedge 2}\otimes A_3^{\wedge 2}),\\
&f^{*11}=A_1\otimes A_3^{\wedge 2},\\
&f^{*12}=A_1^{\wedge 2}\otimes A_3^{\wedge 2}.
\end{align*}
In the above matrices each block that corresponds to an odd splitting appears exactly 4 times. Note that $A_1^{\wedge 2}=A_3^{\wedge 2}=Id_\Z$, but we still write $A_1^{\wedge 2}$ and $A_3^{\wedge 2}$ for clarity.
\end{example}

\begin{proof}[Proof of Theorem~\ref{thm_spheres}]
We can assume that the unstable distribution is oriented. Otherwise we can pass to a double cover as in the proof of Theorem~\ref{thm_product}. Note that the double cover is also a product of spheres of the same dimensions. 

Given an odd splitting $\vec\alpha$ let 
$$
\lambda(\vec\alpha)=\min\{|\mu|\colon \mu\in\mathbb C, \det(A(\vec\alpha)-\mu Id)=0\}
$$
and let 
$$
\lambda=\min_{\vec\alpha\in\EuScript O} \lambda (\vec\alpha).
$$
Consider the set of all odd splittings $\vec\alpha^1, \vec\alpha^2, \ldots\quad \vec\alpha^r$ that achieve this minimum, \ie $\lambda(\vec\alpha^1)= \lambda(\vec\alpha^2)= \ldots= \lambda(\vec\alpha^r)=\lambda$. It is easy to see that the following asymptotic formula holds along a subsequence $\{l_i; i\ge 1\} (l_i\to\infty, i\to\infty)$
$$
Tr(A(\vec\alpha^k)^{-l_i})=q_k\lambda^{-l_i} + o (\lambda^{-l_i}), i\to\infty,
$$
where $q_k$ is an integer given by the ``multiplicity'' of $\lambda$ in $A(\vec\alpha^k)$, $k= 1,\ldots\quad r$.

The trace $Tr(f^{*d})^{-l}$ is the sum of traces of the diagonal blocks of $(f^{*d})^{-l}$. Hence, using Lemma~\ref{lemma_blocks}, we obtain the following expression for the Lefschetz number
\begin{multline*}
\Lambda(f^{l_i})=\sum_{d=0}^{\dim M}(-1)^d Tr(f^{*d})^{-l_i}
=\sum_{\vec\alpha\in\EuScript O}(-1)^{\varepsilon(\vec\alpha)}2^eTr(A(\vec\alpha)^{-l_i})+\text{const}\\
=2^e\biggl(\sum_{k=1}^r(-1)^{\varepsilon(\vec\alpha^k)}q_k\biggr)\lambda^{-l_i}+o(\lambda^{-l_i}), i\to\infty.
\end{multline*}
Let $w=\sum_{k=1}^r(-1)^{\varepsilon(\vec\alpha^k)}q_k$. We obtain
$$
\left|Fix(f^{l_i})\right|=2^e|w|\lambda^{-l_i}+o(\lambda^{-l_i}), i\to\infty. 
$$
Recall that $e\ge 1$ by the assumption of the theorem. If $w\neq 0$ then, by comparing the above asymptotic formula with~(\ref{form_asymp_transitive}), we conclude that $f$ is not transitive.

If $w=0$ then we need to proceed to $\lambda_2$ --- the second smallest absolute value of the eigenvalues of the blocks --- and apply the same argument again. If the coefficient by $\lambda_2^{l_i}$ vanishes as well then we proceed to the third smallest absolute value and so on. Since $|Fix(f^l)|$ grows exponentially and there are only finitely many blocks $A(\vec\alpha)$ each of which has finitely many eigenvalues, this process will terminate in $s$  steps yielding the formula
$$
\left|Fix(f^{l_i})\right|=2^e|w_s|\lambda_s^{-l_i}+o(\lambda_s^{-l_i}), i\to\infty,
$$
where $\lambda_s<1$ and $w_s\neq 0$. Hence $f$ is not transitive.
\end{proof}
\begin{proof}[Proof of Theorem~\ref{thm_spheres2}]
Recall that the odd-dimensional sphere $\S^k=\S^{d_p}$ enters the product exactly once. Let $d\ge 0$. We will write $\vec\alpha\in\D(d)$ to indicate that $\vec\alpha$ is a splitting of $d$. Also we will write $\vec\alpha\in\D_-(d)$ if $\alpha_p=0$ and $\vec\alpha\in\D_+(d)$ if $\alpha_p=1$. Note that given $\vec\alpha\in\D_-(d)$ we can obtain $\vec\alpha_\bullet\in\D_+(d+k)$ by changing $\alpha_p$ from $0$ to $1$. Hence there is a bijective correspondence between the set $\cup_{0\le d\le\dim M}\D_-(d)$ and the set $\cup_{0\le d\le\dim M}\D_+(d)$.

The first block of $f^{*k}$ --- matrix $A_p$ --- is a one-by-one identity matrix. Hence the formula~(\ref{block}) and the block count imply that for any $\vec\alpha\in\D_-(d)$ $\;A(\vec\alpha)=A(\vec\alpha_\bullet)$. We proceed with the calculation of the Lefschetz number.
\begin{multline*}
\Lambda(f^l)=\sum_{d=0}^{\dim M}(-1)^d Tr(f^{*d})^{-l}=\sum_{d=0}^{\dim M}(-1)^d\sum_{\vec\alpha\in\D(d)}A(\vec\alpha)^{-l}\\
=\sum_{d=0}^{\dim M}(-1)^d\biggl(\sum_{\vec\alpha\in\D_+(d)}A(\vec\alpha)^{-l}+\sum_{\vec\alpha\in\D_-(d)}A(\vec\alpha)^{-l}\biggr)\\
=\sum_{d=0}^{\dim M}(-1)^d\biggl(\sum_{\vec\alpha\in\D_+(d)}A(\vec\alpha)^{-l}+\sum_{\vec\alpha_\bullet\in\D_+(d+k)}A(\vec\alpha_\bullet)^{-l}\biggr)\\
=\sum_{d=0}^{\dim M}(-1)^d\sum_{\vec\alpha\in\D_+(d)}A(\vec\alpha)^{-l}-\sum_{d=0}^{\dim M}(-1)^{d+k}\sum_{\vec\alpha\in\D_+(d+k)}A(\vec\alpha)^{-l}\\
=\sum_{d=k}^{\dim M}(-1)^d\sum_{\vec\alpha\in\D_+(d)}A(\vec\alpha)^{-l}-\sum_{d=k}^{\dim M+k}(-1)^{d}\sum_{\vec\alpha\in\D_+(d)}A(\vec\alpha)^{-l}=0
\end{multline*}
We have used the fact that $\D_+(d)=\varnothing$ for $d<k$ and $d>\dim M$. Therefore $Fix(f^l)=\varnothing$ which is a contradiction.
\end{proof}

\section{Betti numbers and Ruelle-Sullivan cohomology classes}
\label{sec_RS}
\begin{theorem}[\cite{RS}] 
\label{tmRS}
Let $f\colon M\rightarrow M$ be a transitive Anosov diffeomorphism of a closed $n$-dimensional manifold $M$ with orientable invariant distributions. Denote by $k$ the dimension of the stable bundle $E^s$. There exist non-zero cohomology classes $s\in H^k(M, \mathbb R)$ and $u\in H^{n-k} (M, \mathbb R)$ such that
\begin{enumerate}
\item 
$f^*s=\pm\lambda s$, $f^*u=\pm\lambda^{-1} u$, where
$\lambda = e^{-{h_{top}(f)}} < 1$.
\item
$s\smallsmile u =\mu$, where $\mu$ is the measure of maximal entropy for $f$.
\end{enumerate}
\end{theorem}

\begin{prop}
\label{prop1}[cf. Proposition~\ref{prop0}]
Let $f\colon M \rightarrow M$ be a codimension $k$ transitive Anosov diffeomorphism with orientable invariant distributions. Then $k$-th Betti number $b_k(M)\geq 2$.
\end{prop}

\begin{proof}[Proof of Proposition~\ref{prop1}] Denote by $f_*^\Z$ and $f_*^\R$ the automorphisms induced by $f$ on integral and real homology respectively. Let $Tor$ be the torsion subgroup of $H_k(M; \mathbb Z)$. Then the integral homology splits

$$H_k(M, \mathbb Z) \simeq \mathbb Z ^p \oplus Tor,$$
By the Universal Coefficients Theorem

$$H_k(M, \mathbb R)\simeq H_k(M, \mathbb Z) \otimes \mathbb R \simeq \mathbb R^p.$$

This isomorphism is natural. Therefore the following diagram commutes
$$
\begin{CD}
(H_k(M; \Z)/Tor)\otimes\R@= H_k(M; \Z)\otimes\R@>>> H_k(M; \R)\\
@V{\hat f_*^\Z\otimes id_\R}VV @V{f_*^\Z\otimes id_\R}VV @V{f_*^\R}VV\\
(H_k(M; \Z)/Tor)\otimes\R@= H_k(M; \Z)\otimes\R@>>> H_k(M; \R)
\end{CD}
$$
Now assume that $p = b_k(M) =1$. Then $\hat f_*^\Z=\pm id$ and from the above diagram we obtain $f_*^\R=\pm id$. This contradicts to the fact that, by naturality of Poincar\'e duality, $\lambda^{-1}>1$ given by Theorem~\ref{tmRS} is an eigenvalue of $f_*^\mathbb R\colon H_k(M; \R)\to H_k(M; \R)$.
\end{proof}

\section{Betti numbers and characteristic classes of invariant distributions}
\label{sec_char_class}
Recall that a characteristic class $c$ is a natural assignment of a cohomology class 
$$c(E)\in H^k(B; \R)$$ 
to each oriented $k$-dimensional, $k\geq0$, vector bundle $p\colon E \rightarrow B$. Naturality means that if $\tilde f\colon E_1 \rightarrow E_2$ is a bundle map that covers $f\colon B_1 \rightarrow B_2$ then
$$
f^*(c(E_2))=c(E_1).
$$

We say that class $c$ has the {\it exponential property} if for any two bundles $p_1\colon E_1\rightarrow B$ and $p_2\colon E_2\rightarrow B$
$$
c(E_1\oplus E_2) =c(E_1)\smallsmile c(E_2).
$$
For further background on characteristic classes we refer to ~\cite{MSt}.

\begin{theorem}
\label{tmMain}
Assume that $f\colon M^n\rightarrow M^n$ is a codimension $k$ transitive Anosov diffeomorphism with orientable invariant distributions. Assume that $c$ is a characteristic class that satisfies the exponential property and $c(TM)\neq 0$. Then the $k$-th Betti number $b_k(M)\geq 3$. 
\end{theorem}

\begin{proof}
Without loss of generality we can assume that $\dim E^s=k$. By the exponential property
$$
c(E^s)\smallsmile c(E^u) = c(E^s \oplus E^u) = c(TM) \in H^n(M; \mathbb R) = \mathbb R.
$$

Since $c(TM)\neq 0$ we have that $c(E^s)\neq 0.$

Choose an orientation on $E^s$. We can assume that $df\colon E^s\rightarrow E^s$ preserves this orientation (otherwise pass to $f^2$). Therefore, by naturality
$$
f^*c(E^s)=c(E^s).
$$
Hence $f^*\colon H^k(M; \mathbb R)\rightarrow H^k(M; \mathbb R)$ has eigenvalue~1. Recall that by Theorem~\ref{tmRS} $f^*$ has an eigenvalue $\lambda \in (0, 1).$ Note that $f^*$ can also be viewed as an automorphism of $H^k(M; \mathbb Z)/Torsion$. Therefore, $det f^*=1$, which implies that $f^*$ also has an eigenvalue $>1$. Hence $\dim H^k(M; \mathbb R)\geq 3$.
\end{proof}

\begin{cor}
\label{cor2}
If a compact oriented manifold $M$ has non-zero Euler characteristic $\chi(M)$ and Betti numbers $b_i(M)\leq 2$, $i\geq 1$. Then $M$ does not admit transitive Anosov diffeomorphisms with orientable invariant distributions. 
\end{cor}

\begin{proof}
Let $e$ be the Euler class (see, {\it e.g.}, ~\cite{MSt}). Then $e(TM)=\chi (M)[M]$ , where $[M]$ is the fundamental class of $M$. Thus Theorem~\ref {tmMain} applies --- if $M$ admits a transitive Anosov diffeomorphism with orientable invariant distributions then at least one Betti number is greater than 2. 
\end{proof}

Observe that any distribution on a simply connected manifold is orientable. Hence Theorem~\ref{thm_euler} follows from Corollary~\ref {cor2}. 

\begin{cor}
Let $\mathbb CP^n$, $n\ge 2$, be the complex projective space. Then manifolds $\mathbb CP^n\times\ S^1$ and $\mathbb CP^n\times \mathbb T^2$ do not admit transitive Anosov diffeomorphisms.
 \end{cor}
\begin{proof}
 Let $M=\mathbb CP^n\times\ \mathbb T^i$ for some $n\ge 2$ and $i\in\{1, 2\}$. Assume that $M$ admits an Anosov diffeomorphism $f$. Since every finite covering of $M$ is a self-covering we can assume that $f$ has orientable invariant distributions. Recall the following formula for the first Pontrjagin class (see, \eg~\cite[Chapter 15]{MSt})
 $$
 p_1(T\mathbb CP^n)=(n+1)a,
 $$
 where $a$ is a generator of $H^2(\mathbb CP^n)$. Since Pontrjagin classes are stable we have $p_1(TM)=p_1(T\mathbb CP^n)\neq 0$. The full Pontrjagin class satisfies the exponential property, hence, by Theorem~\ref{tmMain}, we get that $b_k(M)\ge 3$. However a direct calculation shows that the Betti numbers of $M$ are at most 2.
\end{proof}

\section{Anosov diffeomorphisms on $(2n-1)$-connected $4n$-manifolds.}
Let $M$ be a $(2n-1)$-connected $4n$-dimensional closed oriented manifold. Then $H^i(M;\Z)=0$ for $i\neq 0, 2n, 4n$. Let $[M]$ be the fundamental class of $M$ and let $N$ be the $2n$-th Betti number of $M$. We identify $H^{2n}(M;\Z)$ with $\Z^N$. The {\it intersection form} $Q\colon H^{2n}(M;\Z)\times H^{2n}(M;\Z)\to \Z$ is defined by
$$
Q(x,y)=\langle x\smallsmile y, [M]\rangle.
$$
Form $Q$ is a symmetric bilinear form represented by an $N\times N$ matrix, which we also denote by $Q$ 
$$Q(x,y)=x^tQy.$$ Poincar\'e duality implies that $\det Q=\pm 1$ (see, \eg~\cite{Sc} for the proof of this fact as well as other background on intersection forms).

Form $Q$ has its orthogonal group defined as
$$
SO(Q;\R)=\{A\in SL(n,\R): A^TQA=Q\}.
$$
We will also consider the group $SO(Q;\Z)\subset SO(Q;\R)$ --- the group of integral matrices that are subject to the same condition.
\begin{proof}[Proof of Theorem~\ref{thm_simply_connected}]
Let $f\colon M\to M$ be an Anosov diffeomorphism. Diffeomorphism $f$ induces an automorphism of $H^{2n}(M;\Z)$ given by a matrix $A\in GL(N,\Z)$. After passing to a finite power of $f$ we can assume that
\begin{enumerate}
\item $f$ is orientation preserving;
\item $f$ preserves the orientation of the unstable distribution $E^u$;
\item $A\in SL(N,\Z)$;
\item $A$ does not have eigenvalues that are roots of unity except for 1.
\end{enumerate}
We have
$$
Q(Ax,Ay)=\langle Ax\smallsmile Ay, f^*[M]\rangle=\langle x\smallsmile y, [M]\rangle.
$$
hence $A\in SO(Q;\Z)$.

The Lefschetz formula gives the following expression for the number of points of period $l$
\begin{equation}\label{eq_action_hyperbolic}
|Fix(f^l)|=2+Tr(A^l)
\end{equation}
It follows that $A$ has an eigenvalue of absolute value $>1$.
\begin{remark}
If $Q$ is positive or negative definite then $SO(Q;\R)$ is compact and, hence, $SO(Q;\Z)$ is finite. Therefore $A^m=Id$ for some $m\ge 1$ and $f$ cannot be Anosov in this case.
\end{remark}

Let $V$ be the $A$-invariant $(N-k)$-dimensional subspace of $\R^N$ corresponding  to eigenvalue 1. Consider the orthogonal complement
$$
V^\perp=\{u\in\R^N: Q(u,v)=0\;\; \forall v\in V\}.
$$
It is easy to check that $V^\perp\cap\Z^N\simeq\Z^k$. Also it is easy to see that $V^\perp$ is $A$-invariant. Hence we have the splitting of the form
$$
Q=Q'\oplus Q''\stackrel{\mathrm{def}}{=}Q|_{V^\perp\cap\Z^N}\oplus Q|_{V\cap\Z^N}
$$
and the corresponding splitting of the automorphism. Let $A'\in SO(Q';\Z)$ be the restriction of $A$ to $V^\perp\cap\Z^N$.

By construction $A'$ does not have eigenvalues that are roots of unity. Also recall that, by Poincar\'e duality, $A'=D(A'^T)^{-1}D^{-1}$. Hence all real eigenvalues of $A'$ come in pairs $\lambda$, $\lambda^{-1}$, $\lambda\neq\lambda^{-1}$. Complex eigenvalues come in conjugate pairs. We conclude that $k$ is even. Note that~(\ref{eq_action_hyperbolic}) implies that $k\neq 0$.
Assume that $k=2$. Then, after changing the (integral) basis,  $Q'$ becomes one of the following forms
$$
Q_1=\begin{pmatrix}
1 & 0\\
0 & 1
\end{pmatrix},
Q_2=\begin{pmatrix}
-1 & 0\\
0 & -1
\end{pmatrix},
Q_3=\begin{pmatrix}
1 & 0\\
0 & -1
\end{pmatrix},
Q_4=\begin{pmatrix}
0 & 1\\
1 & 0
\end{pmatrix}.
$$
A direct computation yields
\begin{multline*}
SO(Q_1;\Z)=SO(Q_2;\Z)=\left\{
\begin{pmatrix}
1 & 0\\
0 & 1
\end{pmatrix},
\begin{pmatrix}
-1 & 0\\
0 & -1
\end{pmatrix},
\begin{pmatrix}
0 & -1\\
1 & 0
\end{pmatrix},
\begin{pmatrix}
0 & 1\\
-1 & 0
\end{pmatrix}
\right\},
\\
SO(Q_3;\Z)=SO(Q_4;\Z)=\left\{
\begin{pmatrix}
1 & 0\\
0 & 1
\end{pmatrix},
\begin{pmatrix}
-1 & 0\\
0 & -1
\end{pmatrix}
\right\}
\end{multline*}
Each of these matrices have eigenvalues that are roots of unity. We conclude that $k\ge 4$. Note that since $\chi(M)=N+2\neq 0$, the discussion in Section~\ref{sec_char_class} implies that $N-k\ge 1$. Hence $N\ge 5$ which gives the posited result.
\end{proof}

We would like to remark that the above proof generalizes in a straightforward way to give the following result.
\begin{theorem}
Let $M$ be a simply connected $4n$-dimensional closed manifold. Assume that $M$ has a non-zero Euler characteristic. Also assume that the Betti numbers $b_k(M)\le 1$ for $k\neq 2n$ and $b_{2n}\le 4$. Then $M$ does not support Anosov diffeomorphisms.
\end{theorem}

\bigskip
Andrey Gogolev\\
\smallskip
\address{Mathematical Sciences\\
SUNY Binghamton\\
Binghamton, New York 13902-6000 }\\

Federico Rodriguez Hertz\\
\smallskip
\address{Department of Mathematics\\
 Ê Ê Ê ÊThe Pennsylvania State University\\
 Ê Ê Ê ÊUniversity Park, PA 16802 \\
 Ê Ê Ê ÊUSA}

\end{document}